\documentclass[11pt]{article}
\usepackage{amsfonts, epsfig, amsmath, amssymb, color}
\usepackage{textcomp}
\usepackage{paralist}
\usepackage[linkcolor=black,citecolor=black]{hyperref}

\usepackage[english]{babel}

\renewcommand {\epsilon}{\varepsilon}

\newcommand{\EE}{\mathbb{E}}

\newcommand{\NN}{\mathbb{N}}

\newcommand{\PP}{\mathbb{P}}

\newcommand{\RR}{\mathbb{R}}

\newcommand{\bB}{\mathcal{B}}
\newcommand{\cC}{\mathcal{C}}

\newcommand{\fF}{\mathcal{F}}

\newcommand{\ffF}{\mathfrak{F}}

\newcommand{\fK}{\mathfrak{K}}

\newcommand{\fM}{\mathfrak{M}}

\newcommand{\e}{\varepsilon}
\newcommand{\la}{\lambda}

\newcommand{\si}{\sigma}

\newcommand{\Om}{\Omega}

\newcommand{\mto}{\mapsto}
\newcommand{\ra}{\rightarrow}

\newcommand{\ti}{\tilde}

\newcommand{\lgl}{\ensuremath{\langle}}
\newcommand{\rgl}{\ensuremath{\rangle}}

\newcommand{\lqq}{\leqslant}
\newcommand{\gqq}{\geqslant}

\newtheorem{thm}{Theorem}[section]

\newtheorem{prp}[thm]{Proposition}
\newtheorem{cor}[thm]{Corollary}

\newtheorem{lem}[thm]{Lemma}

\newtheorem{rem}[thm]{Remark}

\DeclareMathSymbol{\ophi}{\mathalpha}{letters}{"1E}

\renewcommand{\phi}{\varphi}

\newcommand{\be}{\begin{equation}}
\newcommand{\ee}{\end{equation}}
\newcommand{\ben}{\begin{equation*}}
\newcommand{\een}{\end{equation*}}

\newcommand{\ba}{\begin{equation}\begin{aligned}}
\newcommand{\ea}{\end{aligned}\end{equation}}

\DeclareMathOperator{\supp}{supp}

\DeclareMathOperator{\dom}{dom}

\newenvironment{proof}{\par\noindent{\bf Proof:}}{\hfill$\blacksquare$\par}

\newfont{\cyrfnt}{wncyr10}
\def\J3{\cyrfnt{\rm \u{\cyrfnt I}}}
\def\j3{\cyrfnt{\rm \u{\cyrfnt i}}}

\usepackage[]{color}
\definecolor{DarkGreen}{rgb}{0.1,0.7,0.3}   

\definecolor{DarkGreen}{rgb}{0.1,0.7,0.3}   

\allowdisplaybreaks[4]

\begin{document}
\title{
A strong averaging principle for L\'evy diffusions \\
in foliated spaces with unbounded leaves.
}

\author{
Paulo Henrique da Costa
\footnote{Departamento de Matem\'atica, Universidade de Bras\'ilia, Bras\'ilia, Brazil; phcosta@unb.br} 
\hspace{1cm}
Michael A. H\"ogele
\footnote{Departamento de Matem\'aticas, Universidad de los Andes, Bogot\'a, Colombia; ma.hoegele@uniandes.edu.co}
\hspace{1cm}
Paulo Regis Ruffino
\footnote{IMECC, Universidade Estadual de Campinas, Campinas, Brazil; ruffino@ime.unicamp.br} 
}

\maketitle

\begin{abstract}

This article extends a strong averaging principle 
for L\'evy diffusions which live on the leaves of a foliated manifold 
subject to small transversal L\'evy type perturbation 
to the case of non-compact leaves. 
The main result states that the existence of $p$-th moments of the foliated L\'evy diffusion for $p\gqq 2$ 
and an ergodic convergence of its coefficients in $L^p$ implies the strong $L^p$ convergence 
of the fast perturbed motion on the time scale $t/\e$ 
to the system driven by the averaged coefficients.  
In order to compensate the non-compactness of the leaves 
we use an estimate of the dynamical system for 
each of the increments of the canonical Marcus equation 
derived in \cite{dCH17}, the boundedness of the coefficients 
in $L^p$ and a nonlinear Gronwall-Bihari type estimate. 
The price for the non-compactness 
are slower rates of convergence, 
given as $p$-dependent powers of $\e$ 
strictly smaller than $1/4$. 

\end{abstract}



\section{Introduction}

The literature on averaging principles for deterministic and stochastic systems 
reaches far back to the 18th century and is enormously rich both in theory and applications. 
At this point, however, we would like to refrain from 
a more systematic review of the long and bifurcated history of the field 
and restrict ourselves to the references to some classical texts. 
Standard texts on the deterministic field include \cite{V-Arnold}, \cite{SVM}, \cite{VM62}, \cite{VM68} and the references therein. 
For stochastic systems we refer to \cite{Freidlin-Wentzell}, \cite{Kabanov-Pergamenshchikov}, \cite{Khasminski-krylov}, \cite{Kifer}, 
\cite{Sowers}, \cite{Namachchvaya-Sowers}, \cite{Ce09} 
and \cite{Borodin-Freidlin} and the respective bibliographies. 

Loosely speaking, an averaging principle describes the 
observation that in a coupled slow-fast system 
in the limit of infinite time scale separation, 
the slow system is close to a system, 
where the fast variable is replaced by the 
limiting measure of its ergodic time average. 
In the case of stochastic differential equations 
rescaling time show that this problem 
can be restated as a problem 
of an ergodic system perturbed by small perturbations. 

The results of this article generalize recent approaches by the authors for diffusions on finite dimensional foliated manifolds. 
For properties of foliated spaces consult \cite{Cannas}, \cite{Ga83}, \cite{Tondeur}, \cite{Walcak}. 
Motivated by \cite{Li} Gargate and Ruffino 
studied in \cite{GR13} the case of foliated Gaussian diffusions on 
compact leaves subject to deterministic Lipschitz transversal perturbation. 
In H\"ogele and Ruffino \cite{HR} the authors treat 
the case of foliated L\'evy jump diffusions with exponential moments 
but still with deterministic transversal perturbation
and compact leaves. This type of processes is described 
in terms of canonical Marcus equations. 

The recent work by da Costa and H\"ogele \cite{dCH17} covers 
the case of a general class of foliated L\'evy diffusions on compact leaves 
perturbed by a near optimally large class of L\'evy diffusions. 
This is carried out with the help of 
a nonlinear comparison principle and 
a fine study of the individual jump increments.  
However in that case the compactness 
still allows global estimates of the horizontal 
components, for instance, in the force acting on the ``vertical'' 
component of the perturbed system. 

This article treats an averaging principle for the same 
type of foliated L\'evy diffusions, however with non-compact leaves. 
The lack of compactness yields an almost unmitigated system 
of fully coupled SDEs. The strategies are once again 
non-linear Gronwall-Bihari type inequalities, 
using the $L^p$ boundedness of the drift.   
However, this comes at the price of slower rates of convergence. 
Our main result, Theorem \ref{thms: main result 2} states that
locally the transversal behavior of $X^{\epsilon}_{\frac{t}{\epsilon}}$ 
can be approximated $L^p$ uniformly in time by the
L\'evy stochastic differential equation in the transversal space with 
coefficients given by the average of the deterministic transversal component of the
perturbation (with respect to the invariant measure on the leaves for the
original unperturbed dynamics) and the diffusion component given by the projection
of the original perturbation into the transversal space. 
We should mention that our results cover the results by \cite{Duan} 
as the special case of uniformly bounded jumps. 

In the Section 2 we present the dynamical and stochastic framework, 
the main hypotheses and the main result. 
In Section 3 we prove the key proposition 
which is the basis for the proof of the main theorem, 
proved in Section 4. 
Wherever possible in the exposition without lost of coherence 
we refer to the article \cite{dCH17} 
in order to avoid trivial repetition. 

\bigskip
\section{Object of study and main results} 
\subsection{The setup } \label{subsec: setup}

The following setup is a non-compact extension of the setup on \cite{dCH17} and \cite{HR}.

\paragraph{The foliated manifold: } 
Let $M$ be a finite dimensional connected, smooth Riemannian manifold. 
It is known by the classical Nash theorem in \cite{Nash} 
that any finite dimensional smooth manifold may be embedded in $\mathbb{R}^{m}$ with $m$ sufficiently large.  
We assume that $M$ is equipped with an $n$-dimensional foliation $\fM$ in the following sense. 
Let $\fM = (L_{x})_{x\in M}$, with $M = \bigcup_{x\in M} L_x$ and the sets $L_x$ are 
equivalence classes of the elements of $M$ satisfying the following. 
\begin{enumerate}
  \item[a.] Given $x_0\in M$ there exist a neighborhood 
$U\subset M$ of the corresponding leaf $L_{x_0}$ and a 
diffeomorphism $\varphi: U \rightarrow L_{x_0}\times V$, 
where $V\subset \RR^d$ is a connected open set 
containing the origin $0\in \RR^d$. 
\item[b.] For any $L_{x_0} \in \fM$ 
the neighborhood $U\supset L_{x_0}$ can be taken small enough such 
that the coordinate map $\varphi$ is uniformly Lipschitz continuous.  
\end{enumerate}

\begin{rem}
The second coordinate of a point $x \in U$, called the 
vertical coordinate, will be denoted with the help the projection $\pi: U \rightarrow V$ by 
$\varphi(x)=(\bar x, \pi(x))$ for some $\bar x\in L_{x}$.  
For any fixed $v\in V$, the preimage $\pi^{-1}(v)$ 
is the leaf $L_x$, where $x$ is any point in $U$ such that the 
vertical projection satisfies $\pi(x)=v$. 
\end{rem}

\paragraph{The unperturbed equation: } 
We are interested in the ergodic behavior of the  
strong solution of a L\'evy driven SDE with jump components which takes 
values in $M$ and which respects the foliation. 
Intuitively, a straight line increment $z$ does not cause the exit from 
the leaf of its current position 
if the entire line segment $(x_0 + \theta z)_{\theta \in [0, 1]}$ is contained in it. 
Ordinary differential equations with a vector field $F$ on the right-hand side 
generalize this concept in the following sense. 
By definition, their solutions follow $F$ as ``infinitesimal'' tangents. 
If $F$ itself is tangential to a given manifold the integral curves 
remain ``infinitesimally tangential'' to the manifold and hence will not leave it. 
Therefore a straight line jump increment $z$ which is transformed in the stochastic integral 
into an integral curve following a tangential  vector field $F$ of a given leaf  
will remain on the leaf, that is, respect the foliated structure of the space. 
This intuition is made rigorous in the notion of stochastic integration in the sense of 
a canonical Marcus equation in the sense of Kurtz, Pardoux and Protter \cite{KPP95}.  
Those equations are the equivalent for L\'evy jump diffusions 
to the Stratonovich equation for Brownian SDE 
in that they satisfy the Leibniz chain rule 
(cf. Proposition 4.2 in \cite{KPP95}). 
Their definition however is different 
since they treat discontinuous processes. 

Let us consider the formal canonical Marcus stochastic differential equation 
\begin{equation}\label{eq: SDE}
d X_t = F_0(X_t) dt + F(X_t) \diamond d Z_t + G(X_t)\circ dB_t, \qquad X_0 = x_0\in M,
\end{equation}
with the following components defined over 
a given filtered probability space $\mathbf{\Omega} = (\Omega, \fF, (\fF_t)_{t\gqq 0}, \PP)$ 
which satisfies the usual conditions in the sense of Protter \cite{Pr04}.

\begin{enumerate}
\item[1.] Let $Z = (Z_t)_{t\gqq 0}$ with $Z_t = (Z^1_t, \dots, Z^r_t)$ be a L\'evy process over $\mathbf{\Om}$ with values in 
$\RR^r$ for some $r\in \NN$ and characteristic triplet $(0, \nu, 0)$. 
It is a consequence of the L\'evy-It\^o decomposition of~$Z$ 
that $Z$ is a pure jump process with respect to a L\'evy measure $\nu: \bB(\RR^r) \ra [0, \infty]$ 
satisfying 
\begin{equation}\label{eq: second order moment}
\int_{\RR^r} (1\wedge \|z\|^2)\; \nu(dz)< \infty\qquad \mbox{ and }\quad \nu(\{0\}) = 0.
\end{equation}
For details we refer to the overview article by Kunita \cite{Ku04} 
and the monographs of Sato \cite{Sa99} or Applebaum \cite{Ap09}.

\item[2.] Let $F \in \cC^2(M; L(\RR^{r}; T\fM))$ satisfying the following. 
The function $x\mapsto F(x)$ is $\cC^2$ and for each $x\in M$ the linear map $F(x)$ 
maps a vector $z\in \RR^r \mto F(x) z \in T_x L_x$ 
to the tangent space of the respective leaf. 
Furthermore, let $F$ and $(D F)F$ be globally Lipschitz continuous on $M$ 
with common Lipschitz constant $\ell>0$. 
 \item[3.] Let $B = (B^1, \dots, B^r)$ be an $\RR^r$-valued Brownian motion on $\mathbf{\Om}$ 
 and $G\in \cC^2(M, L(\RR^r, T\fM))$. We assume that $G$ and $(DG) G$ 
are globally Lipschitz continuous on $M$ with Lipschitz constant $\ell>0$. 
\end{enumerate}

Following \cite{KPP95} a strong solution of the formal equation (\ref{eq: SDE}) 
is defined as a random map $X: [0, \infty) \times \Om \ra M$ 
satisfying almost surely for all $t\gqq 0$ 
\begin{align}\label{eq: SDE unperturbed integral form}
X_t &= x_0 + \int_0^t F_0(X_s) ds + \int_0^t G(X_s) dB_s + \frac{1}{2}\int_0^t (DG(X_s)) G(X_s) d\lgl B\rgl_s \nonumber\\
&\quad+\int_0^t F(X_{s-}) d Z_s + \sum_{0 < s\lqq t} (\Phi^{F \Delta_s Z}(X_{s-})-X_{s-}- F(X_{s-}) \Delta_s Z),
\end{align}
where $\lgl B\rgl_\cdot$ stands for the quadratic variation process of $B$ in $\RR^r$ and 
the function $\Phi^{Fz}(x) = Y(1, x ; Fz)$ and $Y(t, x; Fz)$ for 
the solution of the ordinary differential equation 
\begin{equation}\label{eq: increment ode}
\frac{d}{d\si } Y(\si) = F(Y(\si)) z,  \qquad
Y(0) = x \in M, \quad z\in \RR^r. 
\end{equation}

\paragraph{The perturbed equation: } 
This article studies the situation where an SDE in the sense of (\ref{eq: SDE unperturbed integral form}), 
which is invariant on the leaf of the initial condition $x_0$ is perturbed by a
transversal smooth vector field $\e K dt$ and stochastic differentials 
$\e \ti G \circ d\ti B$ and $\e \ti K \diamond d\ti Z$, $\e>0$, 
in the limit for $\e \searrow 0$. 
More precisely we denote by $X^\e$, $\e>0$ 
the analogous solution in the sense of (\ref{eq: SDE unperturbed integral form}) 
of the perturbed formal system 
\begin{align}
d X^\e_t &= F_0(X^\e_t) dt +  F(X^\e_t) \diamond d Z_t + G(X^\e_t) \circ dB_t \nonumber\\
&\qquad + \e \Big(K(X^\e_t) dt + \ti K(\pi(X^\e_t)) \diamond d\ti Z_t + \ti G(\pi(X^\e_t)) \circ d\ti B_t\Big),\label{eq: SDE perturbed}\\
X^\e_0 &= x_0 \in M,\nonumber
\end{align} 
where the additional coefficients are defined as follows. 
\begin{enumerate}
 \item[4.] The vector field $K: M \ra TM$ is smooth and globally Lipschitz continuous. 
 \item[5.] Let $\ti Z = (\ti Z^1, \dots, \ti Z^r)$ be a L\'evy process on $\mathbf{\Om}$ with 
values in $\RR^r$ with L\'evy triple $(0, \nu', 0)$ for $\nu'$ being a given L\'evy measure. 
The vector field $\ti K \in \cC^2(V, L(\RR^r, TM))$ satisfies that 
$\ti K$ and $(D\ti K)\ti K$ are globally Lipschitz continuous with Lipschitz constant $\ti \ell>0$. 
 \item[6.] Let $\ti B = (\ti B^1, \dots, \ti B^r)$ be a $\RR^r$-valued Brownian motion over $\mathbf{\Om}$ 
 and $\ti G \in \cC^2(V, L(\RR^r, TM))$ satisfy that $\ti G$ and $(D \ti G) \ti G$ are globally Lipschitz continous 
 with Lipschitz constant $\ti \ell>0$. 
 \item[7.] Assume that the stochastic processes $Z, B, \ti Z, \ti B$ are independent on $\mathbf{\Om}$. 
\end{enumerate}

\begin{thm}[\cite{KPP95}, Theorem 3.2 and 5.1]\label{thm: welldefined}
\begin{enumerate}
\item Under the preceding setup 
(items a., b., \mbox{1.- 3.}  and 7.) 
there is a unique $(\fF_t)_{t\gqq 0}$ semimartingale $X$ 
which is a strong global solution of (\ref{eq: SDE}) 
in the sense of equation (\ref{eq: SDE unperturbed integral form}). 
It has a c\`adl\`ag version and is a (strong) Markov process. 
\item Under the preceding setup 
(in particular items a., b. and \mbox{1.-7.}) 
there is a unique semimartingale $X^\e$ 
which is a strong global solution of equation (\ref{eq: SDE perturbed}) in the 
sense of equation (\ref{eq: SDE unperturbed integral form}), 
where $F_0$ is replaced by $F_0 +\e K$ and $F$ by 
$(F, \e \ti K)$, $G$ by $(G, \e \ti G)$, $B$ by $(B, \ti B)$ and $Z$ by $(Z, \ti Z)$. 
The perturbed solution $X^\e$ has c\`adl\`ag paths almost surely and is a (strong) Markov process. 
\end{enumerate}
\end{thm}

\noindent \paragraph{The support theorem: } 
We are now in the position to apply the crucial support theorem, Proposition 4.3, in Kurtz, Pardoux and Protter \cite{KPP95}. 
Under the hypotheses of Theorem \ref{thm: welldefined} 
we have for any $\e>0$ that $x_0\in M$ implies that $\PP(X_t^\e(x_0) \in M ~\forall t\gqq 0) = 1$. 
This result applied to the leaves of $\fM$ yields 
that each solution $X$ of (\ref{eq: SDE}) is \textit{foliated} in the sense 
that $X$ stays on the leaf of its initial condition, i.e. for any $x_0 \in M$ 
we have $\PP(X_t(x_0) \in L_{x_0} ~\forall t\gqq 0) = 1$. 

\subsection{The hypotheses and the main result} 

In the general setup of Subsection \ref{subsec: setup} we assume the following precise hypotheses. 

\paragraph{Hypothesis 1: Integrability.} 
There is an exponent $p\gqq 2$ such that the L\'evy measures $\nu$ of $Z$ and $\nu'$ of $\ti Z$ satisfy
\begin{align*}
&\int_{\RR^r} \|z\|^{p} \,\nu(dz) < \infty \qquad \mbox{ and } \qquad \int_{\RR^r} \|z\|^{2p} \,\nu'(dz) < \infty.
\end{align*}

\paragraph{Hypothesis 2: Foliated invariant measures.}
 \begin{enumerate}
  \item Each leaf $L_{x_0}\in \fM$ passing through $x_0 \in M$ 
has an associated unique invariant measure $\mu_{x_0}$ with $\supp(\mu_{x_0}) =L_{x_0}$ of the unperturbed 
foliated system (\ref{eq: SDE}) with initial condition $x_0$. 
 \item For $v_0 = \pi(x_0)$ the vertical coordinate of $x_0\in M$
we define for $h: M \ra TM$ 
\begin{equation}\label{def: average}
Q^h(v_0) := \int_{L_{x_0}} h(y) \mu_{x_0}(dy).  
\end{equation}
We assume for any globally Lipschitz continuous map $h: M \ra TM$ 
the function 
\begin{equation}\label{eq: average}
\RR^d \supset V \ni v \mapsto Q^{h}(v) \in \RR^d
\end{equation}
is globally Lipschitz continuous. 
\end{enumerate}
\begin{rem}
Note that $L_{x_0}$ only depends on $v_0= \pi(x_0)$. The same is true for $\mu_{x_0}$. 
\end{rem}

\noindent Hypothesis 2 guarantees
that for each  $x_0\in M$, $v_0 = \pi(x_0)\in V$ the stochastic differential equation  
\begin{equation}\label{def: w}
dw_t = Q^{\pi K} \left(w_t \right) dt + \ti K(w_t) \diamond d\ti Z_t + \ti G(w_t) \circ d\ti B_t, \qquad w_0 = v_0\in V
\end{equation}
has a unique strong solution $w = (w_t(v_0))_{t\in [0, \si)}$ on $\mathbf{\Om}$, $\si$ being the 
first exit time of $w$ from~$V$. 

\paragraph{Hypothesis 3: Ergodic convergence of the vertical coefficient in $L^p$.} 

Fix $p\gqq 2$ from Hypothesis 1.  
\begin{enumerate}
 \item There are continuous functions $\eta^0: [0, \infty) \ra [0, \infty)$ and $\bar \eta: M \ra [0, \infty)$, 
where $\eta^0$ is monotonically decreasing with $\eta^0(t) \ra 0$ as $t\ra \infty$ and $\bar \eta$ is globally Lipschitz continuous.  
For all $x_0\in M$ and $t\gqq 0$ we have 
\begin{equation} \label{def: function eta}
\left( \mathbb{E} \left| \frac{1}{t}\int_0^t \pi K (X_s(x_0) )\, ds - 
Q^{\pi K} (\pi (x_0)) \right|^p \right)^{\frac{1}{p}} \lqq \bar \eta(x_0) ~\eta^0(t).
\end{equation}
  \item We assume for any $x_0\in M$ that 
  $\int \bar \eta(y) \mu_{x_0}(dy) < \infty.$
\end{enumerate}
It is known in the literature that there is no standard rate of convergence 
\cite{Kakutani-Petersen}, \cite{Krengel}, which is why we assume an external 
rate of convergence, which decomposes by factors, see for instance \cite{Ku09}.  
 
\bigskip 

\noindent For $\e>0$ and $x_0\in M$ let $\tau^\e$ being the first exit time of 
the solution $X^\e(x_0)$ of equation (\ref{eq: SDE perturbed}) 
from the foliated coordinate neighborhood $U$ of item a) in Subsection \ref{subsec: setup}.  

The main result of this article is the following strong averaging principle. 

\begin{thm}\label{thms: main result 2}
Let Hypotheses 1, 2 and 3 be satisfied for some $p \gqq 2$. 
Then for any $x_0\in M$ and $\la \in (0, \frac{p-1}{p^2})$ 
there are constants $c, C>0$ and $\e_0\in (0,1]$ such 
that $\e\in (0, \e_0]$ and $T\in [0, 1]$ imply 
\begin{align}
\left(\EE\left[\sup_{t\in [0, T\wedge \e \tau^\e \wedge \si]} |\pi\big(X^\e_{\frac{t}{\e}}(x_0)\big) 
- w_t(\pi(x_0))|^p\right]\right)^{\frac{1}{p}} 
\lqq C T \left[ \e^\la + \eta^0(c T|\ln(\e) | ) \right].\label{eq: result NC-exp 2}\\\nonumber
\end{align}
\end{thm}

\begin{rem}
Our results focus on the case with only $p$-th moments, 
hence we set the coefficients $G$ and $\ti G$ to zero 
in the proofs. 
\end{rem}

\bigskip

\section{The transversal perturbations} 

In order to prove the main theorem we need to 
control the error $X^\e-X$ in terms of $L^p$. 
This section is dedicated to the control of this error by the following result. 

\begin{prp}\label{lem: preliminary} 
Let the assumptions of Subsection 2.1 and Hypotheses 1, 2 and~3 be satisfied for some $p\gqq 2$. 
Then for any Lipschitz function $h: M \ra \RR$, $x_0\in M$ 
and for all $T^\cdot: [0, 1]\ra [1, \infty)$  satisfying $\e T^\e \ra 0$
there exist positive constants $\e_0 \in (0, 1]$, $k_1, k_2, k_3>0$  
such that $\e\in (0, \e_0]$ implies
\begin{align}\label{eq: compact bounded}
\left(\EE\left[\sup_{t\in [0, T]}|h(X^\e_t(x_0) ) - h(X_t(x_0))|^p\right]\right)^\frac{1}{p} 
\lqq k_1 \e^{\frac{p-1}{p^2}} \exp(k_2 T).
\end{align}
In addition, the constant $k_1 (x_0) \lqq k_3(1+ \bar \eta(x_0))$.  
\end{prp}

We apply this result for the following setting. 

\begin{cor}\label{cor: preliminary}
Let the assumptions of Proposition \ref{lem: preliminary} be satisfied for some $p\gqq 2$. 
Then for any $\la \in (0,\frac{p-1}{p^2})$ there exist positive constants 
$c_\la$, $\e_0 \in (0, 1]$, $k_4, k_5>0$ 
such that for $T_\e := c_\la |\ln(\e)|$, $\e\in (0, \e_0]$ satisfies  
\begin{equation}\label{eq: cor preliminary}
\left(\EE\left[\sup_{t\in [0, T_\e]}|h(X^\e_t(x_0) ) - h(X_t(x_0))|^p\right]\right)^{\frac{1}{p}} 
\lqq k_4 \e^\la. 
\end{equation}
In addition, the constant $k_4 = k_5 k_1$. 
\end{cor}

\begin{proof}
Plugging $T_\e = - c\ln(\e)$ in the right-hand 
side of (\ref{eq: compact bounded}) we obtain $k_1 \e \exp(k_2 T_\e) = k_1 \e^{\frac{p-1}{p^2}-c k_2}.$
Given $\la \in (0,\frac{p-1}{p^2})$ we fix $c_\la := \frac{1}{k_2} \big(\frac{p-1}{p^2} - \la\big)$ 
and infer the desired result.
\end{proof}

\bigskip

The proof of Proposition \ref{lem: preliminary} 
relies on the following lemma on positive invariant dynamical systems 
and the nonlinear comparison principle 
Corollary \ref{lem: nonlinear comparison} given in the appendix. 
The main difficulty stems from the fact 
that the influence  of the horizontal component 
in the vertical component cannot be estimated 
uniformly by the ``diameter'' of the leaf but has to 
be taken fully into account, which leads to a non-linear comparison principle. 

\begin{lem}\label{lem: invariance lemma}
For $F \in \cC^2(\RR^{r+n}, L(\RR^{r}, \RR^{r+n}))$ being a globally Lipschitz continuous 
matrix-valued vector field and $z\in \RR^r$ denote by $(Y(t; x, Fz))_{t\gqq 0}$ the unique global strong solution of 
the ordinary differential equation 
\[
\frac{d Y}{dt} = F(Y)z\qquad Y(0, x, Fz) = x\in \RR^{r+n}.
\]
\begin{enumerate}
 \item[1)] Then there exists $C>0$ such that for any $z \in \RR^r$ and $x, y \in M$ with $Y(t; x) = Y(t;x, Fz)$ we have 
 \begin{align*}
\sup_{t\gqq 0} 
|(DF(Y(t;x))z)F(Y(t;x))z - (DF(Y(t;y))z)F(Y(t;y))z|
\lqq C~|x-y| ~\|z\|^2.
\end{align*}
 \item[2)] For any $x\in M$ we have $\sup_{t\in [0,1]} \|DF(Y(t;x))F(Y(t;x))\| < \infty.$
\end{enumerate}
\end{lem}
A proof is given in \cite{dCH17} under Lemma 3.1.

\bigskip

\begin{proof} (of Proposition \ref{lem: preliminary})  
The first step of the proof yields the local 
orthogonality of the foliations and a transversal component 
by an appropriate change of coordinates. 
In a second step we estimate the transversal components 
with the help of the ergodic convergence of Hypothesis 3 
and the nonlinear comparison principle Corollary \ref{lem: nonlinear comparison}. 
This is followed by the estimate of the horizontal component 
as the result of a classical Gronwall estimate before we conclude.

\paragraph{1. Change of coordinates: }
We first rewrite $X^\e$ and $X$, the solutions of equations (\ref{eq: SDE}) and (\ref{eq: SDE perturbed}), 
in terms of the coordinates given by the diffeomorphism $\phi$ 
\begin{align*}
(u_t, v_t) := \phi(X_t) \qquad &\mbox{ and }\qquad (u^\e_t, v^\e_t) := 
\phi(X^\e_t).
\end{align*}
The Lipschitz regularities of $h$ and $\phi$ yields for $C_0 := Lip(h \circ \phi^{-1})$ the estimate 
\begin{align}\label{eq: ungleichung 1}
|h(X^\e_t)- h(X_t)| 
\lqq C_0 (|u^\e_t-u_t|+| v^\e_t- v_t|).
\end{align}
The proof of the statement consists in calculating estimates for each summand on the right hand 
side of equation above. We define the 
\begin{align*}
\ffF_0 &:= (D\phi) \circ F_0\circ \phi^{-1},\qquad \ffF := (D\phi) \circ F\circ \phi^{-1},\\
\fK &:= (D\phi) \circ K \circ \phi^{-1}, \qquad \ti \fK := (D\phi) \circ \ti K \circ \phi^{-1}, 
\end{align*}
whose derivatives are uniformly bounded. Considering the components in the image of $\phi$ we have: 
\[
\fK = (\fK_H, \fK_V), \qquad \ti \fK = (\ti \fK_H, \ti \fK_V)
\]
with $ \fK_H, \ti \fK_H \in TL_{x_0}$  and  
$\fK_V, \ti \fK_V \in TV \simeq \RR^d$. 
The chain rule of the canonical Marcus equations 
mentioned in the introduction (Theorem 4.2 of \cite{KPP95})  
yields for equation (\ref{eq: SDE perturbed}) the following form 
in $\phi$ coordinates 
\begin{align}
d u_t^{\e} &=  \ffF_0(u_t^{\e}, v_t^{\e}) dt +  \ffF(u_t^\e, 
v_t^\e) \diamond d Z_t + \e\, \fK_H(u_t^\e, v_t^\e) dt + \e \ti \fK_H(v_t^\e) \diamond d\ti Z_t&\mbox{ with } 
u_t^{\e}\in L_{x_0}, 
\label{eq: foliations SDE in koordinaten}\\
d v_t^{\e} &= \e\, \fK_V(u_t^\e, v_t^\e) dt + \e \ti \fK_V(v_t^\e) \diamond d\ti Z_t&\mbox{ with } v_t^{\e}\in V.
\label{eq: transversale SDE in koordinaten}
\end{align}

\paragraph{2. Estimate of the transversal coordinate $\EE[\sup |v^\e - v|^p]$: }
Identically to \cite{dCH17}, we start with estimates on the transversal components $|v^\e-v|$.   
The change of variables formula $x \mapsto g(x) := |x|^p, x\in \RR^{n+d}$ 
using $\lgl D g(x), u\rangle = p |x|^{p-2}\langle x, u\rgl$ yields almost surely for $t\gqq 0$ 
\begin{align}
|v_t^\e - v_t|^p 
&= p \int_0^t |v_{s}^\e - v_{s}|^{p-2} \lgl v_{s}^\e - v_{s}, \e\fK_V(u^\e_{s}, v^\e_{s}) \rgl ds \nonumber\\
&\qquad + p  \int_0^t |v_{s-}^\e - v_{s-}|^{p-2} \lgl v_{s-}^\e - v_{s-}, \e\ti \fK_V(v^\e_{s-}) \diamond d \ti Z_s\rgl \nonumber \\
&\lqq p \int_0^t |v_{s}^\e - v_s|^{p-1} |\e\fK_V(u_s^\e, v_s^\e)-\e\fK_V(u_s, v_s)|ds \tag{$H_1$}\\
&\quad + p \int_0^t |v_{s}^\e - v_s|^{p-1} |\e\fK_V(u_s, v_s)|ds \tag{$H_2$}\\
&\quad + p \int_0^t |v_{s-}^\e - v_{s-}|^{p-2} |\lgl v_{s-}^\e - v_{s-}, \e(\ti \fK_V(v^\e_{s-})-\ti \fK_V(v_{s-})) d\ti Z_s\rgl| \tag{$H_3$} \\
&\quad + p \int_0^t |v_{s-}^\e - v_{s-}|^{p-2} |\lgl v_{s-}^\e - v_{s-}, \e\ti \fK_V(v_{s-}) d\ti Z_s\rgl| \tag{$H_4$} \\
&\quad + p \sum_{0< s \lqq t} |v_{s-}^\e - v_{s-}|^{p-1} |\Phi^{\e \ti \fK_V \Delta_s \ti Z}(v^\e_{s-})-\Phi^{\e \ti \fK_V \Delta_s \ti Z}(v_{s-})\nonumber\\
&\qquad \qquad - (v^\e_{s-}-v_{s-}) - \e (\ti \fK_V(v^\e_{s-})-\e \ti \fK_V(v_{s-}))\Delta_s \ti Z| \tag{$H_5$}\\[2mm]
&\quad + p \sum_{0< s \lqq t} |v_{s-}^\e - v_{s-}|^{p-1} |\Phi^{\e \ti \fK_V \Delta_s \ti Z}(v_{s-})- v_{s-} - \e \ti \fK_V(v_{s-}) \Delta_s \ti Z| \tag{$H_6$}\\
&=: H_1 + H_2 + H_3 + H_4 + H_5 + H_6. 
\end{align}
\paragraph{2.1 Pathwise estimates: }$\mathbf{H_1:}$ Clearly we have 
\begin{align}
H_1 &\lqq \e p \ell \int_0^t |v_{s}^\e - v_s|^{p} ds. \label{H1}
\end{align}
\noindent $\mathbf{H_2:}$ Young's inequality for the conjugate indices $p$ and $p/(p-1)$ yields 
\begin{align}
H_2 &=\e p \int_0^t |v_{s}^\e - v_s|^{p-1} |\fK_V(u_s, v_s)|ds \nonumber\\
&\lqq \e p  \sup_{[0, t]} |v^\e - v|^{p-1} \int_0^t |\fK_V(u_s, v_s)|ds \nonumber\\
& \lqq \e  \sup_{[0, t]} |v^\e - v|^{p} + \e (p-1) t^p \Big(\frac{1}{t} \int_0^t|\fK_V(u_s, v_s)| ds\Big)^p. \label{H2}
\end{align}
$\mathbf{H_3}$ and $\mathbf{H_4}$: Switching to the Poisson random measure representation 
with respect to the compensated $\ti N'$, for instance see Kunita \cite{Ku04}, we obtain 
\begin{align}
H_3 
&\lqq \e p\int_0^t \int_{\RR^r}|v_{s-}^\e - v_{s-}|^{p-2} \lgl v_{s-}^\e - v_{s-}, (\ti \fK_V(v^\e_{s-})-\ti \fK_V(v_{s-})) z \rgl \ti N'(ds dz) + \e C_1 \int_0^t |v_{s}^\e - v_{s}|^{p}ds.\label{H3}
\end{align}
and 
\begin{align}
H_4 
&\lqq \e p\int_0^t \int_{\RR^r}|v_{s-}^\e - v_{s-}|^{p-2} |\lgl v_{s-}^\e - v_{s-}, \ti \fK_V(v_{s-}) z\rgl| \ti N'(ds dz) + \e C_2 \int_0^t |v_{s}^\e - v_{s}|^{p} ds.\label{H4}
\end{align}
$\mathbf{H_5:}$ For the canonical Marcus terms we apply Lemma \ref{lem: invariance lemma}, statement 1) which yields a positive constant 
such that
\begin{align}
H_5 &\lqq \e^2 C_{3} \int_0^t \int_{\RR^r} |v^\e_{s-}-v_{s-}|^p \|z\|^2 \ti N'(dsdz) 
+\e^2 C_{4}  \int_0^t |v^\e_{s}-v_{s}|^p ~ds.
\label{H5}
\end{align}
The details can be found in \cite{dCH17}. \\
$\mathbf{H_6: }$ For the last term we apply Lemma \ref{lem: invariance lemma}, statement~2), 
and exploit that $\int_{\|z\|>1} \|z\|^4 \nu'(dz) <\infty$, we obtain a positive constant $C_{5}$ such that 
\begin{align}
H_5 &\lqq  
\e^2 C_{5} \int_0^t \int_{\RR^r} |v_{s-}^\e - v_{s-}|^{p-1} \|z\|^4 \ti N'(dsdz)
 + \e^2 C_{6} \int_0^t  |v_{s}^\e - v_{s}|^{p-1}  ds.\label{H6}
\end{align}
Combining the estimates (\ref{H1}- \ref{H6}) we obtain 
\begin{align}
|v_t^\e - v_t|^p
&\lqq \e \sup_{[0, t]} |v^\e - v|^{p} +  \e (p-1) t^p \Big(\frac{1}{t} \int_0^t|\fK_V(u_s, v_s)| ds\Big)^p \label{bounded drift}\\ 
&\quad + \e (C_1 + C_2)\int_0^t |v_{s}^\e - v_{s}|^{p}ds 
\nonumber\\
&\quad + \e^2 C_{4} \int_0^t |v^\e_{s}-v_{s}|^p  ~ds + \e^2 C_{6} \int_0^t  |v_{s}^\e - v_{s}|^{p-1}  ds\nonumber\\
&\quad +\e p \int_0^t \int_{\RR^r} |v_{s-}^\e - v_{s-}|^{p-2} \lgl v_{s-}^\e - v_{s-}, \e\ti \fK_V(v_{s-}) z\rgl| \ti N'(ds dz)\label{RPMI 2}\\
&\quad + \e^2 C_3 \int_0^t \int_{\RR^r} |v^\e_{s-}-v_{s-}|^p \|z\|^2 \ti N'(dsdz) \label{RPMI 4}\\
&\quad + \e^2 p C_{5} \int_0^t \int_{\RR^r} |v_{s-}^\e - v_{s-}|^{p-1} \|z\|^4 \ti N'(dsdz).\label{RPMI 1}
\end{align}

\paragraph{2.2 Estimates on average: } The main difference to \cite{dCH17} is found in the 
treatment of term $H_2$. In the sequel we drop the superscript of $T = T^\e$ where 
$T^\e \in [1, \infty)$ satisfying $\e T^\e \ra 0$. 
Taking the supremum $t\in [0, T]$ and taking the expectation yields that the term (\ref{bounded drift}) can be bounded by 
\begin{align*}
\e~\EE\Big[\sup_{ [0, T]} |v^\e - v|^{p}\Big] +  \e (p-1) C_\infty T^p, 
\end{align*}
where
\begin{equation}\label{def: c infinity}
C_\infty = C_\infty(x_0) = \sup_{t\gqq 0} \EE\Big[\Big(\frac{1}{t}\int_0^t |\fK_V(u_s(x_0), 0)| ds\Big)^p\Big] < \infty 
\end{equation}
due to the convergence 
\[
\EE\Big[\Big(\frac{1}{t}\int_0^t |\fK_V(u_s(x_0), 0)| ds - \int |\fK_V(y, 0)| \mu_{x_0}(dy)\Big)^p\Big]\ra 0, \quad \mbox{ as }t\ra \infty.
\]
This implies in particular that 
\begin{equation}\label{eq: C infinity estimate}
C_{\infty}(x_0) \lqq \int |\fK_V(y, 0)| \mu_{x_0}(dy) + \eta^0(0) \bar \eta(x_0).
\end{equation}
We obtain the integral inequality  
\begin{align*}
\EE\big[\sup_{[0, T]} |v^\e - v|^p\big]
&\lqq \e  C_7\EE\big[\sup_{[0, T]} |v^\e - v|^p\big] + \e (p-1) C_\infty T^p
+ \e C_8\int_0^T \EE\big[\sup_{[0,s]} |v^\e - v|^p\big] ds \\
&\quad + \e C_9\int_0^T \EE\big[\sup_{[0,s]} |v^\e - v|^{p-1}\big] ds \\
&\lqq \e \EE\big[\sup_{[0, T]} |v^\e - v|^p\big] + \e (p-1) C_\infty T^p
+ \e C_8\int_0^T \EE\big[\sup_{[0, s]} |v^\e - v|^p\big] ds \\
&\quad + \e C_9\int_0^T \EE\big[\sup_{[0,s]} |v^\e - v|^{p}\big]^\frac{p-1}{p} ds.
\end{align*}
Hence for any value $\e \in (0,\frac{1}{2}]$ we eliminate the first term 
\begin{align*}
\EE\big[\sup_{[0, T]} |v^\e - v|^p\big]
&\lqq 2 \e  (p-1) C_\infty T^p
+ 2 \e  C_8\int_0^T \EE\big[\sup_{[0,s]} |v^\e - v|^p\big] ds \\
& \quad\quad\quad
+ 2 \e  C_9\int_0^T \EE\big[\sup_{[0,s]} |v^\e - v|^{p}\big]^\frac{p-1}{p} ds.
\end{align*}
That is, for $\Psi(T) = \EE\big[\sup_{[0, T]} |v^\e - v|^p\big]$ we have 
\begin{align*}
\Psi(T) \lqq \e C_{10} T^p + \e C_{11} \int_0^T \Psi(s) ds + \e C_{12} \int_0^T\Psi(s)^\frac{p-1}{p} ds.
\end{align*}
Using the nonlinear extension of the Gronwall-Bihari inequality in Corollary \ref{lem: nonlinear comparison} in the appendix 
essentially given by Pachpatte \cite{Pa}, Theorem~2.4.2, 
which we adapt to our case we obtain a global constant $C>0$ 
such that using that $\e_0 T$ is sufficiently small 
implies for all $\e \in (0, \e_0]$ 
\begin{align}
\Psi(T) \lqq C\big(\e T^p + \e^\frac{p-1}{p} T^{p+\frac{p-1}{p}}\big).\label{eq: vertical}
\end{align}

\paragraph{3. Estimate of the horizontal component $\EE[\sup |u^\e - u|^p]$: }
For convenience of notation we restart with the numbering of constants. 
Formally we obtain 
\begin{align}
u^{\e}_t - u_t & =  \int_0^t (\ffF_0(u_s^{\e}, v_s^{\e}) - \ffF_0(u_s, v_s)) ds 
+ \int_0^t (\ffF(u_{s-}^\e, v_{s-}^\e) -\ffF(u_{s-}, v_{s-})) \diamond d Z_s \nonumber\\
& \qquad + \e\, \int_0^t \big(\fK_H(u_s^\e, v_s^\e) - \fK_H(u_s, v_s)\big) ds + \e \int_0^t \fK_H(u_s, v_s) ds 
+ \e \int_0^t \ti\fK_H(v_{s-}^\e) \diamond d \ti Z_s.
\label{eq: horizontal formal canonical Marcus}
\end{align} 
For further details consult \cite{dCH17}  
where we obtain with the help of 
the change of variable formula for (\ref{eq: horizontal formal canonical Marcus}) 
the following equality in $\RR^{n}$ almost surely for $t\gqq 0$ 
\begin{align}
 |u_t^{\e} - u_t|^p  
 & = p\int_0^t  |u_s^\e- u_s|^{p-2} \lgl u_s^\e- u_s, \ffF_0(u^\e_s, v^\e_s) - \ffF_0(u_s, v_s)\rgl ds \tag{$I_1$} \\
 & \quad + p\int_0^t  |u_{s-}^\e- u_{s-}|^{p-2} \lgl u_{s-}^\e- u_{s-}, (\ffF(u^\e_{s-}, v^\e_{s-}) - \ffF(u_{s-}, v_{s-})) dZ_s\rgl\tag{$I_2$} \\
& \quad + p\sum_{0< s\lqq t}  |u_{s-}^\e- u_{s-}|^{p-2} \lgl u_{s-}^\e- u_{s-}, 
\Phi^{\ffF \Delta_s Z}(u^\e_{s-}, v^\e_{s-}) -\Phi^{\ffF \Delta_s Z}(u_{s-}, v_{s-}) \nonumber\\
& \qquad \qquad \qquad \qquad  -(u^{\e}_{s-} -u_{s-}, v^\e_{s-}-v_{s-}) - (\ffF(u^\e_{s-}, v^\e_{s-})-\ffF(u_{s-}, v_{s-})) \Delta_s Z\rgl\tag{$I_3$}\\
 & \quad + \e\, p\int_0^t |u_s^\e- u_s|^{p-2} \lgl u_s^\e- u_s, \fK_H(u_s^\e, v_s^\e) - \fK_H(u_s, v_s)\rgl ds \tag{$I_4$}\\
 &\quad + \e p\int_0^t |u_s^\e- u_s|^{p-2} \lgl u_s^\e- u_s, \fK_H(u_s, v_s)\rgl ds \tag{$I_5$}\\
 &\quad + \e p \int_0^t |u_{s-}^\e- u_{s-}|^{p-2} \lgl u_{s-}^\e- u_{s-}, \ti \fK_H(v^\e_{s-}) d \ti Z_s \rgl  \tag{$I_6$}\\
  &\quad + p\sum_{0< s\lqq t} |u_{s-}^\e- u_{s-}|^{p-2} \lgl u_{s-}^\e- u_{s-}, \Phi^{\e \ti \fK_H\Delta_s \ti Z}(v_{s-}^\e)-\Phi^{\e \ti \fK_H\Delta_s \ti Z}(v_{s-})\nonumber\\
 &\qquad \qquad - (v^\e_{s-}-v_{s-})- \e (\ti \fK_H(v_{s-}^\e)-\ti \fK_H(v_{s-}))\Delta_s \ti Z\rgl \tag{$I_7$}\\[2mm]
  &\quad + p\sum_{0< s\lqq t} |u_{s-}^\e- u_{s-}|^{p-2} \lgl u_{s-}^\e- u_{s-}, \Phi^{\e \ti \fK_H\Delta_s \ti Z}(v_{s-})
- v_{s-}- \e \ti\fK_H(v_{s-})\Delta_s \ti Z\rgl \tag{$I_8$}\\
  &=: I_1 + I_2 + I_3 + I_4 + I_5 + I_6 + I_7 + I_8. \label{eq: pivot}
\end{align}
In fact, we shall use the following estimate 
\begin{align}
|u_t^{\e} - u_t|^{2p} \lqq 8^{p-1} \sum_{i=1}^8 I_i^2.   
\end{align}

Now, we estimate each of the eight preceding summands on the right-hand side. 
The estimates of $I_1$ and $I_4$ are direct Lipschitz estimates. 
For the stochastic It\^o terms we use the different kinds of maximal inequalities, see for instance \cite{Ap09} and \cite{Ku04}. 
The estimate of the canonical Marcus terms $I_3$, $I_7$ and $I_8$ 
is the most difficult task in which we use 
the result of Lemma \ref{lem: invariance lemma}.  
The term $I_5$ is straightforward. 

\paragraph{3.1 Estimate of the stochastic It\^o integral terms $I_2$ and $I_6$: } 
$\mathbf{I_2}$: Due to the existence of moments of order at least $1$, $I_2$ has the following representation with 
respect to the compensated Poisson random measure associated to $Z$
\begin{align}
&\int_0^t  |u_{s-}^\e- u_{s-}|^{p-2} \lgl u_{s-}^\e- u_{s-}, \big(\ffF(u^\e_{s-}, v^\e_{s-}) - \ffF(u_{s-}, v_{s-})\big) dZ_s\rgl\nonumber\\
&= \int_0^t \int_{\RR^r} 
|u_{s-}^\e- u_{s-}|^{p-2} \lgl u_{s-}^\e- u_{s-}, \big(\ffF(u^\e_{s-}, v^\e_{s-}) - \ffF(u_{s-}, v_{s-})\big) z \rgl \ti N(ds dz)\label{eq: term J1}\\
&\qquad + \int_0^t \int_{\|z\|>1} 
|u_{s}^\e- u_{s}|^{p-2} \lgl u_{s}^\e- u_{s}, \big(\ffF(u^\e_{s}, v^\e_{s}) - \ffF(u_{s}, v_{s})\big) z \rgl \nu(dz) ds.\label{eq: term J2}
\end{align}
For the first term (\ref{eq: term J1}) we exploit the embedding $L^2\subset L^1$, 
Kunita's maximal inequality (see \cite{Ap09} or \cite{Ku04}) for exponent equal to $2$, 
and the Young inequality for the exponents $p/2$ and $p/(p-2)$ 
combined with inequality (\ref{eq: vertical}) and obtain 
\begin{align}
&\EE\left[\sup_{t\in [0, T]}
\Big|\int_0^t \int_{\RR^r} 
|u_{s-}^\e- u_{s-}|^{p-2} \lgl u_{s-}^\e- u_{s-}, \big(\ffF(u^\e_{s-}, v^\e_{s-}) - \ffF(u_{s-}, v_{s-})\big) z \rgl \ti N(ds dz)\Big| \right]^2 \nonumber\\
&\lqq \EE\left[\sup_{t\in [0, T]}
\Big|\int_0^t \int_{\RR^r} 
|u_{s-}^\e- u_{s-}|^{p-2} \lgl u_{s-}^\e- u_{s-}, \big(\ffF(u^\e_{s-}, v^\e_{s-}) - \ffF(u_{s-}, v_{s-})\big) z \rgl \ti N(ds dz)\Big|^2 \right] \nonumber\\
&=  \EE\left[
\int_0^T \int_{\RR^r} 
|u_s^\e- u_s|^{2(p-2)} |\lgl u_s^\e- u_s, \big(\ffF(u^\e_s, v^\e_s) - \ffF(u_s, v_s)\big) z \rgl |^2 \nu(dz) ds\right] \nonumber\\
& \lqq C_1 
\EE\left[\int_0^T \int_{\RR^r} 
|u_s^\e- u_s|^{2(p-1)} \Big(|u^\e_s-u_s|^2 + |v^\e_s - v_s|^2\Big) \|z\|^2 \nu(dz) ds \right]\nonumber\\
& \lqq C_1  ~\Big(\int_{\RR^r} \|z\|^2\nu(dz)\Big) 
~\EE\left[\int_0^T ~\big(|u_s^\e- u_s|^{2p} + |u_s^\e- u_s|^{p-2}|v_s^\e- v_s|^{2}\big)ds \right]\nonumber\\
&\lqq C_2~\Big( \int_0^T  \EE\left[\sup_{[0, s]} |u^\e- u|^{2p} \right] ds+ 
\int_0^T  \EE\left[ |v_s^\e- v_s|^{2p} \right] ds \Big)\nonumber\\
&\lqq C_2\, \bigg(\int_0^T \EE\Big[\sup_{[0, s]} |u^\e- u|^{2p} \Big]ds\bigg)
+  C \big(\e T^{2p} + \e^\frac{2p-1}{2p} T^{2(p+1)+1}\big).
\label{I2 fertig}
\end{align}
The second term follows directly by Young's inequality and the Lipschitz continuity of $\ffF$ 
\begin{align}
&\EE\Big[\sup_{t\in [0,T]} \int_0^t \int_{\|z\|>1} 
|u_{s}^\e- u_{s}|^{p-2} \lgl u_{s}^\e- u_{s}, \big(\ffF(u^\e_{s}, v^\e_{s}) - \ffF(u_{s}, v_{s})\big) z \rgl \nu(dz) ds \Big]^2\nonumber\\
&\lqq \bigg(\ell \int_{\|z\|>1} \|z\| \nu(dz) \EE\Big[\sup_{t\in [0,T]} \int_0^t 
\big(|u_{s}^\e- u_{s}|^{p} + |u_{s}^\e- u_{s}|^{p-1} |v_{s}^\e- v_{s}|\big) ds \Big]\bigg)^2\nonumber\\
&\lqq \bigg(\ell \int_{\|z\|>1} \|z\| \nu(dz)\Big( 2 \int_0^T \EE\Big[\sup_{[0,s]} |u^\e- u|^{p}\Big] ds 
+ \int_0^T  \EE\Big[|v_s^\e- v_s|^{p} \Big] ds\Big)\bigg)^2\nonumber\\
&\lqq C_3 T\int_0^T \EE\Big[\sup_{[0,s]} |u^\e- u|^{2p}\Big] ds+  C \big(\e T^{2p} + \e^\frac{2p-1}{2p} T^{2(p+1)+1}\big).
\label{I2 comp}
\end{align}
$\mathbf{I_6}$: 
We go over to the representation with the Poisson random measure $\ti N'$ associated to
the L\'evy process $\ti Z$ and obtain
\begin{align}
&\sup_{t\in [0, T]} \e \int_0^t |u_{s-}^\e- u_{s-}|^{p-2} \lgl u_{s-}^\e- u_{s-}, \ti \fK_H(v^\e_{s-}) d\ti Z_s\rgl\nonumber \\
&=\sup_{t\in [0, T]} \e \int_0^t \int_{\RR^r}  |u_{s-}^\e- u_{s-}|^{p-2} 
\lgl u_{s-}^\e- u_{s-}, (\ti \fK_H(v^\e_{s-})-\ti \fK_H(v_{s-}))z \rgl \ti N'(ds dz) \tag{$J_1$}\\
&\qquad +\sup_{t\in [0, T]} \e \int_0^t \int_{\|z\|>1} 
|u_{s}^\e- u_{s}|^{p-2} \lgl u_{s}^\e- u_{s}, (\ti \fK_H(v^\e_{s})-\ti \fK_H(v_{s})) z \rgl \nu'(dz) ds\tag{$J_2$}\\
&\qquad +\sup_{t\in [0, T]} \e \int_0^t \int_{\RR^r}  |u_{s-}^\e- u_{s-}|^{p-2} 
\lgl u_{s-}^\e- u_{s-}, \ti \fK_H(v_{s-})z \rgl \ti N'(ds dz) \tag{$J_3$}\\
&\qquad +\sup_{t\in [0, T]} \e \int_0^t \int_{\|z\|>1} 
|u_{s}^\e- u_{s}|^{p-2} \lgl u_{s}^\e- u_{s}, \ti \fK_H(v_{s}) z \rgl \nu'(dz) ds.\tag{$J_4$}
\end{align}
The terms $J_1$ and $J_2$ are estimated analogously to (\ref{I2 fertig}) and (\ref{I2 comp}) 
where $\ffF$ is replaced by $\ti \fK_H$, which yield the following estimates 
\begin{align*}
&\bigg(\EE\Big[\sup_{t\in [0, T]} |\e \int_0^t \int_{\RR^r}  |u_{s-}^\e- u_{s-}|^{p-2} 
\lgl u_{s-}^\e- u_{s-}, (\ti \fK_H(v^\e_{s-})-\ti \fK_H(v_{s-}))z \rgl \ti N'(ds dz)|\Big]\bigg)^2\\
&\lqq C_4 \left(\int_0^T \EE\Big[\sup_{[0,s]} |u^\e- u|^{2p}\Big] ds\right)^\frac{1}{2}+  C \big(\e T^{2p} + \e^\frac{2p-1}{2p} T^{2(p+1)+1}\big)~ 
~ \mbox{ and}\\
&\bigg(\EE\Big[\sup_{[0, T]} \e |\int_0^t \int_{\|z\|>1} 
|u_{s}^\e- u_{s}|^{p-2} \lgl u_{s}^\e- u_{s}, (\ti \fK_H(v^\e_{s})-\ti \fK_H(v_{s})) z \rgl \nu'(dz) ds| \Big]\bigg)^2\\
&\lqq C_5 \int_0^T \EE\Big[\sup_{[0,s]} |u^\e- u|^{2p}\Big] ds +  C \big(\e T^{2p} + \e^\frac{2p-1}{2p} T^{2(p+1)+1}\big).
\end{align*}
For the term $J_3$ we observe that $v_s = 0$ consequently $\fK_V(v_s)$ is constant. 
Applying Kunita's maximal inequality for the exponent $2$, we obtain 
\begin{align*}
&\bigg(\EE\Big[\sup_{t\in [0, T]} \e |\int_0^t \int_{\RR^r}  |u_{s-}^\e- u_{s-}|^{p-2} 
\lgl u_{s-}^\e- u_{s-}, \ti \fK_H(v_{s-})z \rgl \ti N'(ds dz)| \Big]\bigg)^2\\
&\lqq \e^2 \EE \Big[\sup_{t\in [0, T]}  |\int_0^t \int_{\RR^r}  |u_{s-}^\e- u_{s-}|^{p-2} 
\lgl u_{s-}^\e- u_{s-}, \ti \fK_H(v_{s-})z \rgl \ti N'(ds dz)|^2 \Big]\\
&\lqq \e^2 C_6 \int_0^T \int_{\RR^r}  \EE\Big[ |u_{s}^\e- u_{s}|^{2(p-1)} \Big]  \|z\|^2 \nu'(dz) ds\\
&\lqq \e^2 C_6\left(\int_{\RR^r} \|z\|^2 \nu'(dz) \right)
\left(\int_0^T \EE \Big[\sup_{[0,s]} |u^\e- u|^{2p-2}\Big] ds\right)\\
&\lqq \e^2 C_{7} \int_0^T \EE\Big[\sup_{[0,s]} |u^\e- u|^{2p-2}\Big] ds\\
&\lqq \e^2 C_{7} \left(\int_0^T \EE\Big[\sup_{[0,s]} |u^\e- u|^{2p}\Big] ds + \frac{C_{8}}{p} T\right).
\end{align*}
The term $J_4$ is again easier, using $\e T< 1$ and $\e <1$ we obtain 
\begin{align*} 
&\bigg(\EE\Big[\sup_{[0, T]} \e \int_0^t \int_{\|z\|>1} 
|u_{s}^\e- u_{s}|^{p-2} \lgl u_{s}^\e- u_{s}, \ti \fK_H(v_{s}) z \rgl \nu'(dz) ds \Big]\bigg)^2\\
&\lqq \bigg(\e \int_{\|z\|>1} \|z\| \nu'(dz) \|\ti \fK_H(0)\| \int_0^T  \EE[\sup_{[0,s]} |u^\e- u|^{p-1}] ds\bigg)^2 \\
&~ \lqq \e^2 C_{9}\bigg(\int_0^T  \EE\Big[\sup_{[0,s]} |u^\e- u|^{p}\Big] ds\bigg)^2 + C_9 \e^{2p} T^2\\
&~ \lqq \e^2 T C_{9}\int_0^T  \EE\Big[\sup_{[0,s]} |u^\e- u|^{2p}\Big] ds + C_9 \e^{2p} T^2\\
&~ \lqq C_{9}\int_0^T  \EE\Big[\sup_{[0,s]} |u^\e- u|^{2p}\Big]ds + C_9 \e^{2p} T^2.
\end{align*}
Summing up we obtain 
\begin{align}
\EE[\sup_{[0, T]} |I_6|^2] 
&\lqq C_{10} \bigg(\int_0^T \EE\Big[\sup_{[0,s]} |u^\e- u|^{p}\Big] ds 
+ \e^\frac{2p-1}{2p} T^{2(p+1)+1} + \e^2 T\bigg). \label{I6 fertig}
\end{align}

\paragraph{3.2 Estimate of the canonical Marcus terms $I_3$, $I_7$ and $I_8$: }
The estimate is identical to estimate (54) in \cite{dCH17} and yields a constant $C_{11}$ such that 
\begin{align}
|I_3| 
&\lqq 2 C_{11} \Big(\sum_{0< s\lqq t}  |u_{s-}^\e- u_{s-}|^{p} \|\Delta_s Z\|^2 
+ \sum_{0< s\lqq t}  |v^\e_{s-} - v_{s-}|^p \|\Delta_s Z\|^2\Big). \label{eq: Marcus estimate}
\end{align}
Once again, the representation of this sum 
in terms of the Poisson random measure given in Kunita \cite{Ku04} tells us that
\begin{align}
&\sum_{0< s\lqq t} |u^\e_{s-}-u_{s-}|^p \|\Delta_s Z\|^2 \nonumber\\
&= \int_0^t \int_{\RR^r} |u^\e_{s-}-u_{s-}|^p \|z\|^2 \ti N(dsdz) 
+\int_0^t \int_{\|z\|>1} |u^\e_{s}-u_{s}|^p  \|z\|^2 ~\nu(dz)~ds. \label{eq: Marcus estimate 1st term u}
\end{align}
The maximal inequality for integrals with respect to the compensated Poisson random measures 
and inequality (\ref{eq: vertical}) yield  
\begin{align}
\EE[\sup_{[0, T]} |I_3|^2] 
&\lqq C_{12} \int_0^T \int_{\RR^r} \Big(\EE[\sup_{[0, s]} |u^\e-u|^{2p}] 
+ \EE[|v^\e_{s}-v_{s}|^{2p}]\Big) \|z\|^4 ~\nu(dz)~ds\nonumber\\
&= C_{12} \int_{\RR^r} \|z\|^4 \nu(dz) \Big(\int_0^T \Big(\EE[\sup_{[0, s]} |u^\e-u|^{2p}] 
+ \EE[|v^\e_{s}-v_{s}|^{2p}]\Big)  ~ds\Big)\nonumber\\
&\lqq C_{13} \Big(\int_0^T \EE[\sup_{[0, s]} |u^\e-u|^{2p}  ~ds
+ C \big(\e T^{2p} + \e^\frac{2p-1}{2p} T^{2(p+1)+1}\big)\Big).\label{I3 fertig}
\end{align}
$\mathbf{I_7}$: For $I_7$ we apply Lemma \ref{lem: invariance lemma} statement 1) 
and Young's inequality and obtain the analogous result 
\begin{align*}
&\sum_{0< s\lqq t} |u_{s-}^\e- u_{s-}|^{p-2} \lgl u_{s-}^\e- u_{s-}, 
\Phi^{\e \ti \fK_H\Delta_s \ti Z}(v_{s-}^\e)-\Phi^{\e \ti \fK_H\Delta_s \ti Z}(v_{s-})\\
&\qquad - (v^\e_{s-}-v_{s-}) - \e (\ti \fK_H(v_{s-}^\e)-\ti \fK_H(v_{s-}))\Delta_s \ti Z\rgl \\[2mm]
&\lqq \e^2 C_{14} \Big(\sum_{0< s\lqq t}  \big(|u_{s-}^\e- u_{s-}|^{p} +  |v^\e_{s-} - v_{s-}|^p\big) \|\Delta_s \ti Z\|^2\Big).
\end{align*}
Rewriting the last expression in terms of the (compensated) Poisson random measure $\ti N'$ we obtain 
\begin{align}
&\sum_{0< s\lqq t}  \big(|u_{s-}^\e- u_{s-}|^{p} +  |v^\e_{s-} - v_{s-}|^p\big) \|\Delta_s \ti Z\|^2 \nonumber\\ 
&= \int_0^t \int_{\RR^r} \big(|u_{s-}^\e- u_{s-}|^{p} +  |v^\e_{s-} - v_{s-}|^p\big) \|z\|^2 \ti N'(dsdz) \label{stoch int 1}\\
&\qquad +\int_0^t \int_{\|z\|>1} \big(|u_{s}^\e- u_{s}|^{p} +  |v^\e_{s} - v_{s}|^p\big) \|z\|^2 \nu'(dz) ds.\label{stoch int 2}
\end{align}
Kunita's maximal inequality for the exponent $2$  yields
\begin{align*}
&\EE\Big[|\sup_{[0, T]} \int_0^t \int_{\RR^r} \big(|u_{s-}^\e- u_{s-}|^{p} +  |v^\e_{s-} - v_{s-}|^p\big) \|z\|^2 \ti N'(dsdz)|^2\Big]\\
&\lqq C_{15} \int_0^T \int_{\RR^r} \EE\Big[|u_{s}^\e- u_{s}|^{2p} +  |v^\e_{s} - v_{s}|^{2p}\Big] \|z\|^2 \nu'(dz) ds\\
&\lqq C_{16} \left(\int_{\RR^r} \|z\|^2 \nu'(dz) 
\int_0^T \EE\Big[\sup_{[0, s]} |u^\e- u|^{2p}\Big]ds  +  \int_0^T \EE\Big[|v^\e_s - v_{s}|^{2p}\Big] ds\right)\\
&\lqq C_{16} \int_{\RR^r} \|z\|^2 \nu'(dz) 
\int_0^T \EE\Big[\sup_{[0, s]} |u^\e- u|^{2p}\Big]ds  + C_{17} \big(\e T^{2p} + \e^\frac{2p-1}{2p} T^{2(p+1)+1}\big), 
\end{align*}
where $C_{17} = C$ from (\ref{eq: vertical}). The term (\ref{stoch int 2}) is treated obviously such that 
\begin{align}
\EE[\sup_{[0, T]} |I_7|^2] \lqq \e^2 C_{18} \int_{\RR^r} \|z\|^4 \nu'(dz) 
\int_0^T \EE\Big[\sup_{[0, s]} |u^\e- u|^{2p}\Big] ds + C_{17} \big(\e T^{2p} + \e^\frac{2p-1}{2p} T^{2(p+1)+1}\big).\label{I7 fertig}
\end{align}
$\mathbf{I_8}$: For $I_8$ Lemma \ref{lem: invariance lemma}, statement 2), yields 
\begin{align*}
&\sum_{0< s\lqq t} |u_{s-}^\e- u_{s-}|^{p-2} \lgl u_{s-}^\e- u_{s-}, \Phi^{\e \ti \fK_H\Delta_s \ti Z}(v_{s-})- v_{s-}
 - \e \ti \fK_H(v_{s-})\Delta_s \ti Z\rgl \\
&\lqq \sum_{0< s\lqq t} |u_{s-}^\e- u_{s-}|^{p-1} |\Phi^{\e \ti \fK_H\Delta_s \ti Z}(v_{s-})- v_{s-}
 - \e \ti \fK_H(v_{s-})\Delta_s \ti Z| \\ 
&\lqq \e^2 C_{19}\sum_{0< s \lqq t} |u_{s-}^\e- u_{s-}|^{p-1} \|\Delta_s \ti Z\|^2,
\end{align*}
such that again Kunita's inequality with exponent $2$ and 
elementary Young's estimate for parameters $\frac{p-2}{p}$ and $\frac{p}{2}$ yield 
\begin{align}
\EE[\sup_{[0, T]} |I_8|^2]
&\lqq  \e^2 C_{20} \int_{\RR^r} \|z\|^4 \nu'(dz) \int_{0}^T \EE[\sup_{[0, s]} |u^\e- u|^{2p-2}\Big] ds\nonumber\\
&\lqq  \e C_{21} \int_{0}^T \EE[\sup_{[0, s]} |u^\e- u|^{2p}\Big] ds + C_{21} \e^\frac{p}{2} T.\label{I8 fertig}
\end{align}

\paragraph{3.3 Estimate of $I_5$: }
\begin{align*}
\int_0^T |u_s^\e- u_s|^{p-2} \lgl u_s^\e- u_s, \e \fK_H(u_s, v_s)\rgl ds 
&\lqq C_{22} \int_0^T \e |u_s^\e- u_s|^{p-1} ds \\
&\lqq C_{22} \int_0^T \e |u_s^\e- u_s|^{p} ds + C_{22} \e^p T
\end{align*}
such that 
\begin{align}
\EE[\sup_{[0,T]} |I_5|^2] \lqq \e C_{23} T \int_0^T \EE\Big[\sup_{[0,s]}|u^\e- u|^{2p}\Big] ds+ C_{23} \e^{2p} T^2.\label{I5 fertig}
\end{align}

\paragraph{3.4 Linear comparison principle: } ~

\noindent Taking the supremum and the expectation of the left-hand side of equation (\ref{eq: pivot}) 
and combining the estimates of 
$8^{p-1} \sum_{i=1}^8 \EE[\sup_{[0,T]} |I_i|]$ 
given by (\ref{I2 fertig}), (\ref{I2 comp}), (\ref{I6 fertig}), (\ref{I3 fertig}), (\ref{I7 fertig}), (\ref{I8 fertig}) 
and (\ref{I5 fertig}) we obtain a positive constant  $C_{24}$
\begin{align*}
\EE\left[\sup_{[0, T]}|u^{\e} - u|^{2p}\right] 
&\lqq C_{24} \bigg(\int_{0}^T \EE\left[\sup_{[0, s]}|u^{\e} - u|^{2p}\right] ds + \e T^{2p} + \e^\frac{2p-1}{2p} T^{2(p+1)+1}+ (\e^{p} T)^2\bigg)\\
&\lqq C_{24} \bigg(\int_{0}^T \EE\left[\sup_{[0, s]}|u^{\e} - u|^{2p}\right] ds + \e T^{2p +3}\bigg).
\end{align*}
Finally 
\begin{align*}
\EE\left[\sup_{[0, T]}|u^{\e} - u|^{p}\right] 
&\lqq \EE\left[\sup_{[0, T]}|u^{\e} - u|^{2p}\right] \lqq C_{25} \e^\frac{2p-1}{2p} T^{2p +3} e^{C_{25} T}. 
\end{align*}

\paragraph{4. Conclusion: } 
The estimates of the sum of the vertical and the horizontal estimate yield 
\begin{align*}
\EE\Big[\sup_{[0, T]} |h(X^\e)- h(X)|^p\Big] 
&\lqq C_0 
\Big(\EE\Big[\sup_{[0, T]} |u^\e-u|^p\Big]
+\EE\Big[\sup_{[0, T]} |v^\e-v|^p\Big]\Big)\\
&\lqq C_{26} \e^\frac{2p-1}{2p} T^{2p +3} e^{C_{25} T} + C\big(\e T^p + \e^\frac{p-1}{p} T^{p+\frac{p-1}{p}}\big)\\ 
&\lqq C_{27} \e^\frac{p-1}{p} e^{C_{27} T}.
\end{align*}
We finally note that the only dependence on the initial conditions stems from $C_\infty$ 
and hence by (\ref{eq: C infinity estimate}) the estimate $C_{27} \lqq C_{28} (1+ \eta^0(0) \bar \eta(x_0))$. 
This finishes the proof. 
\end{proof}

\bigskip

\section{The averaging error and the proof of the main result}

For convenience we fix the following notation. 
Given $h: M\ra \RR^n$ a globally Lipschitz continuous function and 
$Q^h: V \ra \RR^n$ its average on the 
leaves defined in definition (\ref{def: average}). 
For $t\gqq 0$, $x_0\in M$ and $\e\in (0,1]$ we write 
\[
\delta^h_{x_0}(\e, t) := \int_{0}^{t \wedge \e\tau^{\e}} 
h(X^\e_{\frac{s}{\e}}(x_0)) - Q^h(\pi(X^\e_{\frac{s}{\e}}(x_0))) ds.
\]

\begin{prp}\label{prp: componente vertical}
Let the assumptions of Proposition \ref{lem: preliminary} be satisfied for fixed $p\gqq 2$. 
Then for any globally Lipschitz continuous function $h: M\ra \RR^n$, $\la\in (0,\frac{p-1}{p^2})$ and $x_0 \in M$ 
there exist constants $b_1>0$ and $\e_0\in (0,1]$ such that for $\e\in (0, \e_0]$ and $T \in [0, 1]$ we have 
\begin{align*}
\left(\EE\left[\sup_{s \in [0, T]} |\delta_{x_0}^h(\e,s)|^p \right]\right)^\frac{1}{p} 
\lqq b_1 T \left[ \e^{\la}  +   \eta^0 \left( c T |\ln \epsilon | \right) \right],
\end{align*}
where $c := \frac{1}{k_2} \big(\frac{p-1}{p^2} - \la\big) \wedge \ell Lip(\phi^{-1})$ 
is given in Corollary \ref{cor: preliminary}. 
and $\eta^0$ is the temporal factor of 
the ergodic rate of convergence 
given in equation (\ref{def: function eta}) by Hypothesis 3.
\end{prp}

\paragraph{Proof of Proposition (\ref{prp: componente vertical}) : } 
Fix $x_0\in M$. For $\e\in (0,1)$ and $T> 0$ we define the partition 
\[
t_0 = 0 < t_1^\e < \dots < t_{N^\e}^\e \lqq \frac{T}{\e} \wedge \tau^{\e}
\]
with the following step size 
\begin{equation*}
\Delta_\e := - c T \ln(\e)\qquad \mbox{ for some }c>0.
\end{equation*}
The grid points of the partition are given by $t_n^\e := n \Delta_\e\wedge \tau^\e $ for $0 \lqq n \lqq N_\e$ for $\e \in (0, 1]$ 
with $N_\e = \lfloor \frac{1}{c\e |\ln(\e)|} \rfloor$. 
The term $\delta^h_{x_0}(\e, t)$ can be estimated by the following three sums 
\begin{equation}\label{eq: delta decomposition}
|\delta^h_{x_0}(\e, T)| \lqq |A_1(T, \e)| + |A_2(T, \e)| +|A_3(T, \e)|,  
\end{equation}
where 
\begin{align*}
A_1(T, \e) &:= \e \sum_{n=0}^{N_\e} \int_{t_n}^{t_{n+1}} [h(X^\e_{s}(x_0)) -h(X_{s-t_n}(X^\e_{t_n}(x_0)))] ~ds, \\[2mm]
A_2(T, \e) &:= \e \sum_{n=0}^{N_\e} \int_{t_n}^{t_{n+1}} 
[h(X_{s-t_n}(X^\e_{t_n}(x_0))) - \Delta_\e Q(\pi(X^\e_{t_n}(x_0)))] 
~ds, \\[2mm]
A_3(T, \e) &:=\sum_{n=0}^{N_\e}\e \Delta_\e Q(\pi(X^\e_{t_n}(x_0))) - \int_{0}^{t_{N_\e+1}} \e Q(\pi(X^\e_{\frac{s}{\e}}(x_0))) ~ds. \\[2mm]
\end{align*}
The following lemmas estimate the preceding terms one-by-one. 
For convenience of the reader we number the constants $C_i$. 

\begin{lem}\label{Lemma1a}
For any $\la\in (0,\frac{p-1}{p})$ there exist positive constants $b_2>0$ and $\e_0\in (0,1]$ 
such that for any $\e\in (0, \e_0]$ and $T\gqq 0$ 
\begin{align*}
\left(\EE\left[\sup_{s\in [0, T]} |A_1(s, \e)|^p \right]\right)^{\frac{1}{p}} \lqq b_2 T \e^\la.
\end{align*}
\end{lem}

\begin{proof}
Using the Markov property analogously to \cite{dCH17} and Corollary \ref{cor: preliminary} we obtain 
\begin{align*}
&\EE\Big[\sup_{[0, T]}A_1(s, \e)|^p\Big]^\frac{1}{p} \\
&= \e \sum_{n=0}^{N_\e-1} \EE\Big[\EE\Big[|\int_{t_n}^{t_{n+1}} |h(X^\e_{s}(x_0)-h(X_{s-t_n}(X^\e_{t_n}(x_0))ds|^p~|~\fF_{t_n}\Big]\Big]^\frac{1}{p} \\
&= \e \sum_{n=0}^{N_\e-1} \EE\Big[\EE\Big[|\int_{t_n}^{t_{n+1}} |h(X^\e_{s-t_n}(y))-h(X_{s-t_n}(y)ds|^p~|~y=X^\e_{t_n}(x_0)\Big]\Big]^\frac{1}{p} \\
&\lqq \e (N_\e +1) \Delta_\e \max_{n=0, \dots, N_\e} \EE\Big[\EE\Big[|\sup_{s\in [0, t_1]} |h(X^\e_{s}(y))-h(X_{s-t_n}(y)|^p~|~y=X^\e_{t_n}(x_0)\Big]\Big]^\frac{1}{p}\\
&\lqq T \e^\la \max_{n= 0, \dots, N_\e} \EE\big[k_4(X^\e_{t_n}(x_0))\big].  
\end{align*}
Note that by Corollary (\ref{cor: preliminary}) we have 
\[
\max_{n= 0, \dots, N_\e} \EE\big[k_4(X^\e_{t_n}(x_0))\big] \lqq k_3 k_5\big(1+  \max_{n= 0, \dots, N_\e}\EE\big[\bar \eta(X^\e_{t_n}(x_0))\big]\big).
\]
It remains to bound the last summand. We estimate as follows for any $n\in \NN$ 
\begin{align*}
&\EE\big[\bar \eta(X^\e_{t_n}(x_0))\big] \\
&\lqq \EE\big[\bar \eta(X^\e_{t_n}(x_0)) - \bar \eta(X_{t_n}(x_0))\big] + \EE\Big[\bar \eta(X_{t_n}(x_0))\Big]\\
&\lqq \bar \ell\, \EE\big[|X^\e_{t_n}(x_0) - X_{t_n}(x_0)|\big] + \int \bar \eta(y) \mu_{x_0}(dy) + 
\sup_{t\gqq 0} \EE\Big[|\int \bar \eta(y) \mu_{x_0}(dy) - \bar \eta(X_{t}(x_0))|\Big]\\
&\lqq \bar \ell\, \EE\big[|X^\e_{t_n}(x_0) - X_{t_n}(x_0)|^p\big]^\frac{1}{p} + C_1.
\end{align*}
For the first term in the preceding expression we derive a recursion formula. 
Using Theorem 3.2 in Kunita \cite{Ku04} it yields for the horizontal component 
\[
\EE\Big[\sup_{s\in 0, T]} |X_s(x_1) - X_s(x_2)|^p \Big] \lqq e^{\ell Lip(\phi^{-1}) T} |x_1 - x_2|^p,
\]
which implies the inequality 
\[
\EE\Big[|X_{t_1}(x_1) - X_{t_1}(x_2)|^p \Big] \lqq C_2 \e^\la |x_1 - x_2|^p. 
\]
We estimate 
\begin{align*}
&\EE\big[|X^\e_{t_n}(x_0) - X_{t_n}(x_0)|^p\big]^\frac{1}{p} \\
&\lqq \EE\big[|X^\e_{t_n}(x_0) - X_{t_n-t_{n-1}}(X^\e_{t_{n-1}}(x_0))|^p\big]^\frac{1}{p} 
+ \EE\big[|X_{t_n}(X^\e_{t_{n-1}}(x_0)) - X_{t_n}(x_0))|^p\big]^\frac{1}{p}\\
&\lqq C_3 \e^\la \EE\big[k_4(X^\e_{t_{n-1}}(x_0))\big] + C_2 \e^\la \EE\big[|X^\e_{t_{n-1}}(x_0) - X_{t_{n-1}}(x_0)|^p\big]^\frac{1}{p}\\
&\lqq C_4 \e^\la \Big(1 + \EE\big[\bar \eta(X^\e_{t_{n-1}}(x_0))\big]\Big)+ C_2 \e^\la \EE\big[|X^\e_{t_{n-1}}(x_0) - X_{t_{n-1}}(x_0)|^p\big]^\frac{1}{p}\\
&\lqq C_4 \e^\la \Big(C_1 + \bar \ell\, \EE\big[|X^\e_{t_{n-1}}(x_0) - X_{t_{n-1}}(x_0)|^p\big]^\frac{1}{p}\Big)
+ C_2 \e^\la \EE\big[|X^\e_{t_{n-1}}(x_0) - X_{t_{n-1}}(x_0)|^p\big]^\frac{1}{p}\\
&= C_5 \e^\la \EE\big[|X^\e_{t_{n-1}}(x_0) - X_{t_{n-1}}(x_0)|^p\big]^\frac{1}{p} +C_6 \e^\la.
\end{align*}
That is, for $\psi_n := \EE\big[|X^\e_{t_n}(x_0) - X_{t_n}(x_0)|^p\big]^\frac{1}{p}$ we have then 
\begin{align*}
\psi_n \lqq  C_7 \e^\la \psi_{n-1} + C_7 \e^\la, 
\end{align*}
which gives the following estimate for any $n\in \NN$ and $k\in \{1, \dots, n\}$ 
\begin{align*}
\psi_n \lqq  (C_7 \e^\la )^{n-k}  + \sum_{i=1}^{n-k} (C_7 \e^\la )^i. 
\end{align*}
For $C_7 \e_0^\la  < \frac{1}{2}$ we obtain for any $n\in \NN$ the estimate 
\begin{align*}
\psi_n 
&\lqq  C_7 \e^\la  + \sum_{i=1}^{\infty} (C_7 \e^\la)\lqq  3 C_7 \e^\la < \infty .
\end{align*}
Under these assumptions, we obtain for all $\e \in (0, \e_0]$ 
\begin{align}
\max_{n= 0, \dots, N_\e}\EE\big[\bar \eta(X^\e_{t_n}(x_0))\big] 
&\lqq \max_{n= 0, \dots, N_\e} \Big(\bar \ell\, 
\EE\big[|X^\e_{t_n}(x_0) - X_{t_n}(x_0)|^p\big]^\frac{1}{p} + C_1\Big) 
\lqq C_7 \e^\la + C_1 < \infty.
\label{eq: estimate initial value}
\end{align}
Going back to our main estimate, we obtain $C_8>0$ such that 
\begin{align*}
\EE\Big[\sup_{s\in [0, T]}A_1(s, \e)|^p\Big]^\frac{1}{p} 
&\lqq T \e^\la  k_3 k_5 \Big(1 +  \max_{n= 0, \dots, N_\e}\EE\big[\bar \eta(X^\e_{t_n}(x_0))\big]\Big) \lqq C_8 T \e^\la .
\end{align*}
\end{proof}

\begin{lem}\label{Lemma2b} 
For any $\la\in (0,\frac{p-1}{p})$ there exist positive constants $b_3>0$ and $\e_0\in (0,1]$ 
such that for any $\e\in (0, \e_0]$ and $T\gqq 0$ 
\begin{align*}
\left(\EE\left[\sup_{s\in [0, T]} |A_2(s, \e)|^p \right]\right)^{\frac{1}{p}} \lqq b_3 T \eta^0(c T|\ln(\e)|). 
\end{align*}
\end{lem}

\begin{proof} 
We have
\begin{eqnarray*}
 \left(\EE\left[\sup_{s\in [0, T]} |A_2(s, \e)|^p 
\right]\right)^{\frac{1}{p}}    & \lqq & 
 \epsilon\  
  \left[\mathbb{E}  \left|
 \sum^{N_\e-1}_{n=0}[
 \int^{t_{n+1}}_{t_n}{h(X_{s-t_n}(X^{\epsilon}_{t_n}(x_0)))ds}- \Delta_\e 
 Q^h (\pi (X^{\epsilon}_{ t_n}(x_0))) ]  \right|^p\right]^{\frac{1}{p}} \\
 && \\
& \lqq  & \epsilon \Delta_\e  \sum^{N_\e-1}_{n=0}
 \left[\mathbb{E}  \Big|
\frac{1}{\Delta_\e }
\int^{t_{n+1}}_{t_n}{h (X_{s-t_n}(X^{\epsilon}_{t_n}(x_0)))ds}- 
Q (\pi (X^{\epsilon}_{ t_n}(x_0)))  \Big|^p\right]^{\frac{1}{p}}.
\end{eqnarray*}
We apply the Markov property for all $n=0, \ldots, N_\e$. 
By Hypothesis 3 the two terms inside the
modulus converge to each other when 
$\Delta_\e$ goes to infinity with rate of 
convergence bounded by $\bar \eta(X^\e_{t_n}(x_0)) \eta^0 (\Delta_\e )$. 
Hence, for small $\epsilon$
we have
\begin{align*}
 \left(\EE\left[\sup_{s\in [0,  T]} |A_2(s, \e)|^p \right]\right)^{\frac{1}{p}}
& \lqq \e N_\e \, \Delta_\e \, \eta^0(\Delta_\e) 
\max_{n=0, \dots, N_\e} \EE[\bar \eta(X^{\epsilon}_{ t_n}(x_0))] \\
& \lqq  T \eta^0(c T|\ln(\e)|) \max_{n=0, \dots, N_\e}  \EE[\bar \eta(X^{\e}_{t_n}(x_0))].
\end{align*}
Therefore, using (\ref{eq: estimate initial value}), we obtain for $\e \in (0, \e_0]$ the estimate 
\begin{align*}
 \left(\EE\left[\sup_{s\in [0,  T	]} |A_2(s, \e)|^p \right]\right)^{\frac{1}{p}}
&\lqq  C_8 T \eta^0(c T|\ln(\e)|) .
\end{align*}
\end{proof}

\begin{lem}\label{Lemma2c}
For any $\la\in (0,\frac{p-1}{p})$ there exist positive constants $b_4>0$ and $\e_0\in (0,1]$ 
such that for any $\e\in (0, \e_0]$ and $T\gqq 0$ 
\begin{align*} 
\left(\EE\left[\sup_{s\in [0, T]} |A_3(s,\e)|^p \right]\right)^{\frac{1}{p}} \lqq 
b_4 T \e^\la.
\end{align*}
\end{lem}

\begin{proof} 
We calculate 
\begin{align}
|A_3(T,\e)| &=  \Big| \sum_{n=0}^{N_\e} \e \Delta_\e Q^{\pi K}(\pi(X^\e_{t_n}))
- \int_{0}^{N_\e \Delta_\e} \e Q^{\pi K}(\pi(X^\e_{\frac{s}{\e}})) ~ds \Big|\nonumber\\
& \lqq \e \sum_{n=0}^{N_\e} \Delta_\e \sup_{t_n \lqq s < t_{n+1}}  
|Q^{\pi K}(\pi(X^\e_s)) - Q^{\pi K}(\pi(X^\e_{t_n}))|\nonumber\\
& \lqq \e \Delta_\e C_1 \sum_{n=0}^{N_\e} \sup_{t_n \lqq s < t_{n+1}} |v^\e_s - v^\e_{t_n}|. \label{eq: A3 estimate}
\end{align}
By Minkowski's inequality, the Markov property, Proposition \ref{prp: componente vertical} 
and (\ref{eq: estimate initial value}) (with the appropriate constant $C_8$) we have hat 
\begin{align*}
\EE[\sup_{s\in [0, T]} |A_3(s,\e)|^p]^\frac{1}{p}
& \lqq T C_1 \max_{n\in \{0, \dots, N_\e\}} 
\EE[\EE[\sup_{t_n \lqq s < t_{n+1}} |v^\e_{s-t_{n}}(y) - v^\e_{0}(y)|^p~|~y = \fF_{t_n}]]^\frac{1}{p}\\
& \lqq T C_1 \max_{n\in \{0, \dots, N_\e\}} 
\EE[\EE[\sup_{t_0 \lqq s < t_{1}} |v^\e_{s}(y) - v^\e_{0}(y)|^p~|~y = X_{t_n}^\e(x_0)]]^\frac{1}{p}\\
& \lqq T C_2 \e^\la \EE\Big[k_4(X^\e_{t_n}(x_0))\Big]  \\
& \lqq T C_2 C_8 \e^\la.
\end{align*} 
\end{proof}

This ends the proof of Proposition \ref{prp: componente vertical}.

\paragraph{Proof of the main Theorem \ref{thms: main result 2}: } 
With the help of Proposition \ref{prp: componente vertical}, the proof of 
Theorem \ref{thms: main result 2} is identical the one given in Section 5 of \cite{dCH17}. 

\bigskip

\section{Appendix: Nonlinear comparison principle}\label{sec: appendix1}

\begin{prp}[Pachpatte \cite{Pa}]\label{prop: Pachpatte}
Let $u, f, g$ and $h$ be nonnegative continuous functions defined on $\RR^+$. 
Let $v$ be a continuous non-decreasing subadditive and submultiplicative function 
defined on $\RR^+$ and $v(u) >0$ on $(0, \infty)$. 
Let $e, \phi$ be continuous and nondecreasing functions defined on $\RR^+$ 
with $p$ being strictly positive and $\phi(0) = 0$. 
If 
\begin{align*}
u(t) \lqq e(t) + g(t) \int_0^t f(s) u(s) ds + \phi\Big(\int_0^t h(s) v(u(s)) ds\Big)  
\end{align*}
for all $t\gqq 0$, then for any $0 \lqq t \lqq t_2$ 
\begin{align*}
u(t) \lqq a(t) \Big[e(t) + \phi\Big(F^{-1}\big(F(A(t)) + \int_0^t h(s) v(a(s)) ds \big)\Big)\Big],
\end{align*}
where 
\begin{align*}
a(t) &:= 1+ g(t) \int_0^t f(s) \exp\Big(\int_s^t g(\si) f(\si) d\si\Big)ds,\\
A(t) &:= \int_0^t h(s) v(a(s) e(s)) ds,\\
F(t) &:= \int_{0}^t \frac{ds}{v(\phi(s))}, 
\end{align*}
$F^{-1}$ is the inverse of $F$ and $t_2\in \RR^+$ such that 
\[
F(A(t)) +  \int_0^t h(s) v(a(t)) ds \in \dom(F^{-1}) \qquad \mbox{ for all }0\lqq t\lqq t_2.
\]
\end{prp}

In the following special case of coefficients 
it is possible to drop the continuity assumption on $u$.

\begin{cor}\label{lem: nonlinear comparison}
Let $\Psi$ a non-negative, measurable, increasing function 
and $h$ be nonnegative, continuous, increasing 
function on the interval $[0, T]$ satisfying 
for $p\gqq 2$, $\e>0$, $c>0$ and any $t\in [0, T]$ 
the inequality
\begin{equation}\label{eq: nonlinear comparison}
\Psi(t) \lqq \e c t^p + \e c \Big(\int_0^t \Psi(s) + \Psi(s)^\frac{p-1}{p} ds\Big), \qquad t\in [0, T]. 
\end{equation}
Then there is a constant $k>0$ such that for any $\e_0 \in (0, 1]$ such that $\e_0 T < k$ we have for all 
$t\in [0, T]$ and $\e \in (0, \e_0]$ 
\begin{align*}
\Psi(t) \lqq C \Big(\e t^p+t^p  (\e t)^\frac{p-1}{p}\Big).
\end{align*}
\end{cor} 

\begin{proof} 
For $e(t) = c  \e t^p$, $g\equiv 1$, $f, h\equiv \e c $, $\phi(t) = t$, $w(t) = t^\frac{p-1}{p}$ we calculate 
the coefficients of Proposition \ref{prop: Pachpatte}
\begin{align*}
a(t) &:= 1+ \e c \int_0^t \exp(\e c (t-s)) ds 
= \exp(\e c t)
\end{align*}
and in the limit of $\e t$ being small $(\e t\ll 1)$ we have  
\begin{align*}
\e \int_0^t a(s)^\frac{p-1}{p} ds 
&= \e t \Big(\frac{\exp(c \frac{p-1}{p} \e t) -1}{c \frac{p-1}{p} \e t}\Big)\lqq_{\e t\ll 1} 2 \e t.
\end{align*}
Applying the change of parameter $r=\e s$ it follows that
\begin{align*}
A(t) & :=  \int_0^t  \exp(\e c \frac{p-1}{p} s) (e(s))^\frac{p-1}{p} ds =  \int_0^t \exp( c \frac{p-1}{p} \e s ) (c \e s^p)^\frac{p-1}{p} ds \\
&=  \e^\frac{p-1}{p}\int_0^{\e t} \exp( c \frac{p-1}{p} r ) c^\frac{p-1}{p}\left(\frac{r}{\e}\right)^{p-1} \frac{dr}{\e} 
\lqq t \frac{1}{\e^p t}\int_0^{\e t} \exp( c \frac{p-1}{p} r ) (c r)^\frac{p-1}{p} dr \\
&\lqq_{\e t \ll 1} 2 t \exp( c \frac{p-1}{p} \e t ) (c \e t)^\frac{p-1}{p} \lqq C_1 t \exp( c \frac{p-1}{p} \e t ) (\e t)^\frac{p-1}{p}.
\end{align*}
Finally, we obtain
\[
F(t) := \int_{0}^t s^{-\frac{p-1}{p}} ds = p r^\frac{1}{p} \qquad \mbox{ and }\quad F^{-1}(t) := \frac{t^p}{p^p}.
\]
In the sequel we follow the proof of Theorem 2.4.2 in Pachpatte \cite{Pa} 
and define the continuous, positive, non-decreasing function 
\[
n(t) := e(t) + \phi\Big(\int_0^t h(s) w(u(s)) ds\Big) = e(t) + \e c \int_0^t h(s) u(s)^\frac{p-1}{p} ds, \qquad t\gqq 0,  
\]
such that inequality (\ref{eq: nonlinear comparison}) can be restated as 
\begin{align*}
u(t) \lqq n(t) + g(t) \int_0^t f(s) u(s) ds = e(t) + \e c \int_0^t u(s) ds . 
\end{align*}
It is well-known, see for instance \cite{Am90}, that this 
integral estimate implies the following Gronwall-Bellmann 
inequality also in the case of $u$ being merely positive measurable. 
The main reason is that the integral is absolutely continuous with a bounded density. 
This result yields 
\[
u(t) \lqq a(t) n(t), \qquad t\gqq 0.  
\]
The remainder of the proof of Theorem 2.4.2 in \cite{Pa} 
does use the continuity of $u$ and remains intact. 
\end{proof}

\bigskip 
 
\section{Acknowledgements: } 
The author PHC would like to thank the Department of Mathematics of Brasilia University for providing support.
The authors MAH and PRR would like express his gratitude 
for the hospitality received at the Departameto de Matem\'atica 
at Universidade de Bras\'ilia and the IMECC at 
UNICAMP in February 2018. 
The funding of MAH by the FAPA project 
``Stochastic dynamics of L\'evy driven systems'' 
at the School of Science at Universidad de los Andes 
is greatly acknowledged. 
The author PRR is partially supported by Brazilian CNPq 
proc. nr. 305462/2016-4,  by FAPESP proc. nr.
2015/07278-0 and 2015/50122-0.


\bigskip 

\begin{thebibliography}{99}


\bibitem{Am90} Amann, H.:
\newblock  Ordinary differential equations. An introduction to nonlinear analysis.
\newblock  de Gruyter Studies in Mathematics, 13. Walter de Gruyter \& Co., 1990.


\bibitem{Ap09} Applebaum, D.: L{\'e}vy processes and stochastic calculus. {Cambridge university press}, 2nd edition, 2009.

\bibitem{V-Arnold} Arnold, V.: { Mathematical Methods in Classical
Mechanics}. {Springer}, 2nd edition, 1989.

\bibitem{Borodin-Freidlin} Borodin, A., Freidlin, M.: Fast oscillating random
perturbations of dynamical systems with conservation laws. { Ann. Inst. H.
Poincar\'e. Prob. Statist.} 31, 485-525 (1995).  

\bibitem{Cannas} Cannas, A.: Lectures on Symplectic Geometry. { Lecture Notes in Mathematics} 1764, 2008.

\bibitem{Ce09} Cerrai, S.: A Khasminskii type averaging principle for stochastic reaction-diffusion equations.
Ann. Probab. 19(3), 899-948 (2009).

\bibitem{dCH17} da Costa, P.H., H\"ogele, M.A.: { Strong averaging along foliated L\'evy diffusions with heavy tails on compact leaves}. 
Potential Analysis 47(3), 277-311 (2017). 

\bibitem{Duan} Xu, Y., Duan, J., Xu, W.: An averaging principle for stochastic dynamical systems with Lvy
noise, Physica D 240,  1395-1401 (2011).

\bibitem{Freidlin-Wentzell} Freidlin, M.I., Wentzell, A.D.: { Random 
Perturbations of Dynamical Systems}. Springer-Verlag, 1991.

\bibitem{GR13} Gargate, I.I.G., Ruffino, P.R.: 
{\ An averaging principle for diffusions in foliated spaces}. 
Ann. Probab. 44 (1), 567-588 (2016).

\bibitem{Ga83} Garnett, L.:
{ Foliation, the ergodic theorem and Brownian motion}, 
Journal of Functional Analysis 51, 285-311 (1983).

\bibitem{HR} H\"ogele, M.A., Ruffino, P.R.: Averaging along foliated L\'evy diffusions.
Nonlinear Analysis 112, 1-14 (2015).

\bibitem{Kabanov-Pergamenshchikov} Kabanov, Y., Pergamenshchikov, S.: {
Two-Scale Stochastic Systems: asymptotic analysis and control}.
Springer-Verlag, 2003.

\bibitem{Kakutani-Petersen} Kakutani, S., Petersen, K.: The speed of 
convergence in the ergodic theorem. { Monat. Mathematik} 91, 11-18 (1981).

\bibitem{Khasminski-krylov} Khasminski, R., Krylov, N.: On averaging
principle for diffusion processes with null-recurrent fast component. { Stoch
Proc Appl.}  93, 229-240 (2001).

\bibitem{Kifer} Bakhtin, V., Kifer, Y.: Nonconvergence examples in averaging. Geometric and probabilistic structures in dynamics. Contemp. Math. 469, 1-17 (2008).

\bibitem{Krengel} Krengel, U.: 
On the speed of convergence of the ergodic theorem. {
Monat. Mathematik} 86, 3-6 (1978).

\bibitem{KPP95} Kurtz, T.G., Pardoux, E., Protter, Ph.: {
Stratonovich stochastic differential equations driven by general semimartingales.}
Annales de l'H.I.P., section B 31(2), 351-377 (1995).

\bibitem{Ku09}
Kulik, A.: 
\newblock Exponential ergodicity of the solutions of SDE's with a jump noise.
\newblock Stochastic Processes and their Applications 119,  602-632 (2009). 

\bibitem{Ku04} Kunita, H.: {
Stochastic differential equations based on L\'evy processes and stochastic flows of diffeomorphisms}.
In: Rao, M.M. (ed.) Real and Stochastic Analysis. Birkh\"auser, 305-373 (2004).

\bibitem{Li} Li, X.-M.: An averaging principle for a completely
integrable stochastic Hamiltonian systems. { Nonlinearity} 21, 
803-822 (2008).

\bibitem{Namachchvaya-Sowers} Namachchvaya, S., Sowers, R.: Rigorous stochastic averaging at a center with additive noise. Meccanica 37,  85-114 (2002).

\bibitem{Nash} Nash, J.: The imbedding problem for Riemannian manifolds. Annals of Mathematics 63 (1), 20-63 (1956).

\bibitem{Pa} Pachpatte, B.G.: Inequalities for differential and integral equations. 
{ Academic Press }, 1998.

\bibitem{Pr04} Protter, Ph.: Stochastic integration and differential equations. 
{ Springer-Verlag}, 2004.

\bibitem{Sa99}
Sato, K.-I.:
\newblock L\'evy processes and infinitely divisible distributions.
\newblock { Probab. Theory Relat. Fields 111, 287--321} (1998).

\bibitem{Sowers} Sowers, R.: Stochastic averaging with a flattened Hamiltonian: a Markov process on a stratified space (a whiskered sphere). 
Trans. Am. Math. Soc 354,  853-900 (2002).

\bibitem{SVM} Sanders, J.A., Verhulst, F., Murdock, J.:  Averaging
Methods in Nonlinear dynamical Systems. Springer, 2nd edition, 2007.

\bibitem{Tondeur} Tondeur, P.: {Foliations on Riemannian manifolds}.
Universitext, Springer-Verlag, 1988.

\bibitem{VM62} Volsov, V.M.: Some types of calculation connected with averaging in the theory of non-linear 
vibrations. 
{USSR Computational Mathematics and Mathematical Physics}, 3(1), 1962.

\bibitem{VM68} Volsov, V.M., Morgunov, B.I.: Methods of calculating stationary resonance 
vibrational and rotational motions of certain non-linear systems. 
{USSR Computational Mathematics and Mathematical Physics} 8(2), 1968.

\bibitem{Walcak} Walcak, P.: { Dynamics of foliations, groups and
pseudogroups}. Birkh\"auser Verlag, 2004.

\end{thebibliography}
\end{document}